\documentclass[reqno]{amsart}

\usepackage{amssymb}
\usepackage{amsmath}
\usepackage{mathrsfs}

\theoremstyle{plain}
\newtheorem{thm}{Theorem}[section]
\newtheorem{cor}[thm]{Corollary}
\newtheorem{lemma}[thm]{Lemma}

\theoremstyle{definition}

\newtheorem{rmk}[thm]{Remark}

\numberwithin{equation}{section}

\DeclareMathOperator{\Gal}{Gal}

\DeclareMathOperator{\ord}{ord}
\DeclareMathOperator{\Hom}{Hom}
\DeclareMathOperator{\HH}{H}

\DeclareMathOperator{\rank}{rank}
\DeclareMathOperator{\rnk}{rnk}

\DeclareMathOperator{\Spec}{Spec}

\def\Cal{\mathcal}
\def\Bb{\mathbb}

\def\bP{{\Bb P}}

\def\Z{{\Bb Z}}

\def\G{{\Bb G}}

\def\s{{\sigma}}
\def\e{{\epsilon}}

\def\cO{{\Cal O}}
\def\cC{{\mathscr C}}
\def\cD{{\Cal D}}

\def\cAg{{\Cal A_{g,1}}}
\def\cM{{\Cal M}}
\def\cH{{\Cal H}}

\def\T{{T}}

\def\Proj{{\bP_k^1}}

\def\aa{{\bf n}}
\def\n{{n}}

\def\Om{{\Omega_C^1}}
\def\Diffs{{\HH^0(C, \Om)}}

\def\OmX{{\Omega_X^1}}
\def\DiffsX{{\HH^0(X, \OmX)}}
\def\HdRX{{\HH_{\rm dR}^1(X)}}

\def\Diff#1{{\cD_{#1}}}

\def\set#1{\left\{#1\right\}}

\begin{document}

\title[The rank of the Cartier operator]
{The Rank of the Cartier operator on cyclic covers of the projective line}

\author{Arsen Elkin}

\address{Colorado State University, Fort Collins, Colorado 80523, United States}
\curraddr{Institute of Mathematics, University of Warwick, Coventry CV47AL, United Kingdom}
\email{A.Elkin\char`\@warwick.co.uk}



\date{\today}

\begin{abstract}
We give a bound on the $a$-numbers of the Jacobian varieties 
of Kummer covers of the projective line in terms of their ramification data
and the characteristic of the base field.
\end{abstract}

\thanks{The author was partially supported by the Marie Curie Incoming International Fellowship PIIF-GA-2009-236606}

\maketitle

%=========================================================================
\section{Introduction}\label{sec:intro}
%=========================================================================
Let $k$ be an algebraically closed field of characteristic $p>0$,
and, for an integer $g\geq 1$, let $\cAg$ denote the moduli stack
of principally polarized abelian varieties of dimension $g$ defined over $k$.
Space $\cAg$ can be stratified according to the invariants of the
kernel of multiplication by $p$ map
of the abelian varieties, considered as a group scheme, 
such as $p$-rank, Newton polygon, and Ekedahl-Oort type.
These stratifications can then be induced onto any subfamily of abelian varieties.
For example, the points of the moduli space $\cM_g$
of smooth projective curves of genus $g$ defined over $k$
can be considered as a locus in $\cAg$ via the Torelli morphism and
thereby $\cM_g$ can be stratified by invariants of the $p$-torsion
of the Jacobian varieties of such curves.

One of such invariants is the {\em $a$-number} $a_X$ of an abelian variety $X$,
defined as
$$
a_X := \dim_k \Hom(\alpha_p, X[p]),
$$
where $\alpha_p$ is the kernel of the Frobenius endomorphism
on the group scheme $\G_a$
($\alpha_p \cong \Spec(k[x]/x^p)$ as a scheme),
and the group scheme $X[p]$ is the kernel of multiplication by $p$ on $X$.
When $X=J(C)$, the Jacobian variety of a curve $C$, we write $a_C$
instead of $a_{J(C)}$ and refer to it as the $a$-number of $C$.

The $a$-number plays the following role in the Ekedahl-Oort stratification of $\cAg$,
which is known \cite{Oort} to be the stratification by
the isomorphism classes of the $p$-torsion.
For an abelian variety $X$ of dimension $g$ there exists \cite{Oda} an exact sequence
$$
0 \to \DiffsX \to \HdRX \to \DiffsX \to 0,
$$
via which one can consider $\DiffsX$
to be a vector subspace of the de Rham cohomology space $\HdRX$. 
If $V$ is the Verschiebung operator on $\HdRX$,
then $\DiffsX = V\,\HdRX$ 
is the $g$-dimensional space of the canonical filtration \cite{Geer}
of $\HdRX$. Therefore, if $\nu_X = [\nu_{X,1}, \ldots, \nu_{X,g}]$ denotes the 
final type of $X$,
then $\nu_{X,g} := \dim(V\,\DiffsX)$. 
It is known \cite[(5.2.8)]{LO} that $a_X$ equals 
the dimension of the kernel of the action of Verschiebung on $\DiffsX$ 
and therefore $$\nu_{X,g} = g_X - a_X.$$
Of a particular interest are abelian varieties $X$ for which $g_X=a_X$ or, equivalently,
whose Ekedahl-Oort final type is $[0, \ldots, 0]$. These are
precisely \cite[1.6]{LO} the abelian varieties that are {\em isomorphic} to
a self-product of a supersingular elliptic curve, and are known as the
{\em superspecial} abelian varieties.

For an integer $0 \leq a \leq g$, let us write $\T_{g,a}$ for the locus in $\cAg$
of points corresponding to abelian varieties $X$ with $a_X \geq a$
or, equivalently, $\nu_{X,g} \leq g - a$ in the final type of $X$.
For example, $T_{g,g}$ is the locus of superspecial abelian varieties.
A result of Re \cite{Re} states that if
$$
g - a < \frac{2g}{p(p+1)} + \frac{p-1}{p+1},
$$
then $\cM_g \cap \T_{g,a}$ is empty.
Let $\cH_g$ denote the locus in $\cAg$ of points corresponding
to Jacobian varieties of hyperelliptic curves of genus $g$ defined over $k$.
Re also proves that if
$$
g-a \leq \frac{g}{p} - \frac{p+1}{2p},
$$
then $\cH_g \cap \T_{g,a}$ is empty.

In this paper, we reexamine these results for the subloci of $\cAg$ corresponding to
Jacobian varieties of Kummer covers of the projective line $\Proj$
and, in particular, obtain improved bounds for such curves, including the hyperelliptic curves.
We make extensive use of the fact \cite{Oda} that
for a smooth projective curve $C$ the action of Verschiebung on
$\HH^0(J(C), \Omega_{J(C)}^1) \cong \Diffs$
coincides with the action of the Cartier operator \cite{Cartier} on $\Diffs$.

Let $\pi: C \to \Proj$ be a smooth branched cyclic cover of order $\n$ relatively prime to $p$.
Denote by $\alpha_1, \ldots, \alpha_{r} \in \Proj(k)$ the branch points of $\pi$.
Then $C$ is the projectivization of an affine curve given by equation
$$
y^\n = \prod_{\substack{j=1 \\ \alpha_j\neq \infty}}^{r} (x-\alpha_j)^{n_j}
$$
for some integers $0 < n_j < \n$.  If $\alpha_j = \infty$ for some $1\leq j \leq r$,
then there is a unique choice of $0 < n_j < \n$ so that 
$$
\sum_{j=1}^{r} n_j \equiv 0 \pmod \n.
$$
If $\pi$ is not ramified over infinity, this congruence is automatically satisfied.
If $C$ is connected, which we assume to be the case
throughout this paper, then the greatest common divisor $(\n, n_1, \ldots, n_{r})=1$.
The integers $n_1, \ldots, n_{r}$ are well-defined up to a permutation
and simultaneous multiplication by a unit in $\Z/\n\Z$.
The tuple $\aa =(\n; n_1, \ldots, n_{r})$
is referred to  \cite{BouwDeg, BouwThesis} as the {\em (ramification) type} of $\pi$, or of $C$.
The genus $g_C$ of $C$ is given \cite[(4)]{Koo} by the formula
$$
g_C = \frac12 \n (r-1) -  \frac12 \sum_{j=1}^{r} (\n, n_j) + 1.
$$
Given a valid ramification type $\aa$, we write $\cH_{\Proj, \aa}$ for
the locus in $\cAg$ of points corresponding to the Jacobian varieties
of smooth cyclic covers of $\Proj$ of type $\aa$, where $g$ 
is the genus of curves of type $\aa$.
For example, $\cH_{\Proj, \aa}=\cH_g$ for
$
\aa := (2; \underbrace{1, \ldots, 1}_{\text{$2g+2$ times}}).
$
\medskip

Fix a primitive $\n$th root of unity $\zeta \in k$. Then there exists a generator
$\delta$ of $\Gal(C/\Proj)$ acting on the closed points of $C$ by
$(x,y)\mapsto(x,\zeta^{-1} y)$.
By functoriality, $\delta$ induces an automorphism
$\delta^*$ of the vector space $\Diffs$, and
\begin{equation}\label{eqn:directsum}
\Diffs = \bigoplus_{i=0}^{\n-1} \Diff{i},
\end{equation}
where $\Diff{i}$ denotes the $\zeta^i$-eigenspace of $\delta^*$.
The dimensions $d_i := \dim(\cD_i)$ of these spaces are given by
\begin{equation}\label{eqn:multiplicity}
d_i = \sum_{j=1}^{r} \left\langle {i n_j}/{\n} \right\rangle - 1.
\end{equation}
In particular, $d_0=0$.
Here, $\lceil \cdot \rceil$ denotes the ceiling function, 
$\lfloor \cdot \rfloor$ denotes the floor function, 
and $\langle \cdot \rangle$ is the fractional part.
For example, $\langle m/n \rangle n$ is the reduction of $m$ modulo $n$.

Finally, the Cartier operator satisfies
$\cC(\zeta^{pi} \omega) = \zeta^{i} \cC(\omega),$
and therefore
$$
\cC(\cD_{\langle pi/n \rangle n}) \subset \cD_{i}
$$
for all $1 \leq i < \n$. Since $(p, \n)=1$, the function $i \mapsto\langle pi/n \rangle n$
is a permutation on the set $\set{1,\ldots, n-1}$, and we denote its inverse by $\s$.
In other words, $\cC(\cD_i) \subset \cD_{\s(i)}$.

Our main result is
\begin{thm}\label{thm:main}
Let $\pi: C\to \Proj$ be a Kummer cover of type
$(\n; n_1, \ldots, n_{r})$ defined over
an algebraically closed field of characteristic $p>0$.
Define $d_i$ and $\s(i)$ as above. Then
$$
\sum_{i=1}^{\n-1}
\min\left(2 \lfloor {d_i}/{p}\rfloor, 
d_{\s(i)} \right)
\leq g_C - a_C \leq
\sum_{i=1}^{\n-1} \min(d_i, d_{\s(i)}).
$$
\end{thm}

As a particular case, we almost double Re's bound for hyperelliptic curves,
for which $\cD_1 = \Diffs$, and so $d_1 = g$.
\begin{cor}
If $p>2$ and integers $0 \leq a \leq g$ satisfy
$$
g-a \leq \frac{2g}{p}-2,
$$
then $\cH_g \cap \T_{g,a}$ is empty.
\end{cor}

In addition, we obtain the following result.
\begin{cor}\label{cor:superspecial}
Let $C$ be a superspecial Kummer cover of $\Proj$ of ramification type $(\n; n_1, \ldots, n_{r})$.
Then for each $1\leq i \leq \n-1$, one has $\dim(\cD_i) < p$.
In particular, $g_C \leq (p-1)(\n-1)$.
\end{cor}
For $n<(p+1)/2$ this bound is an improvement over the bound $g \leq p(p-1)/2$
in \cite{Ekedahl}. Even though the latter result is sharp for smooth curves in general,
examples of sharpness come from certain Fermat curves that fall outside the
class of curves we consider.

We also prove
that the $a$-number of a Kummer curve in characteristic $2$
depends only on the ramification type of the curve, and not on the location
of the branch points.
%mirroring a similar result \cite{EP} for hyperelliptic curves in this characteristic.
\begin{cor}\label{cor:chartwo}
Let $k$ be an algebraically closed field of characteristic $p=2$,
and $C$ be a Kummer cover of $\Proj$ of ramification type $(\n; n_1, \ldots, n_{r})$
defined over $k$.
Then
$$
 a_C = \sum_{i=1}^{n-1} \min(d_{i}, d_{\s(i)}).
$$
\end{cor}

This article is structured as follows. 
Section \ref{sec:regular} contains a review of the structure of the space 
$\Diffs$ of regular differential forms on $C$.
In Section \ref{sec:cartier}, we compute the action of the Cartier operator
operator on $\Diffs$. Finally, Section \ref{sec:rank} contain the proof
of Theorem \ref{thm:main} and Corollary \ref{cor:superspecial}

The author would like to thank Dr. Rachel Pries for her guidance and insightful
discussions on the subject of this work, and the referee for helpful suggestions
and further inspiration.

%=========================================================================
\section{Regular differential forms}\label{sec:regular}
%=========================================================================
First, we examine the action of the group $\mu_\n(k)$ of $n$th roots of unity in $k$ on $\Diffs$. 
The entire contents of this section can be found in
\cite{BouwThesis} and \cite{Koo}, but  for clarity we restate them here without proofs.

Let $\pi: C \to \Proj$ be a connected branched cyclic cover of order $\n$ 
(not necessarily prime) that is relatively prime to the characteristic $p$
of an algebraically closed field of definition $k$ of $C$.
Let $B\subset \Proj(k)$ be the branching locus of the map $\pi$.
It is a well-known fact of the Kummer theory that $C$ has an affine model given
by an equation $y^\n = f(x)$, where $f\in k(x)$. After a change of variables,
we may assume, without loss of generality, that $f\in k[x]$ and
\begin{equation*}%\label{eqn:formf}
f(x) = \prod_{\alpha\in B-\set{\infty}} (x-\alpha)^{n_\alpha}
\end{equation*}
for some integers $1 \leq n_\alpha < \n$. 
If $\pi$ is not ramified over $\infty \in \Proj(k)$, then
\begin{equation}\label{eqn:typecong}
\sum_{\alpha\in B} n_\alpha \equiv 0 \pmod{\n}.
\end{equation}
The converse is also true.
Therefore, if $\infty \in B$, we can find a unique integer $1 \leq n_\infty < \n$
such that congruence \eqref{eqn:typecong} is satisfied.
The tuple $(\n; n_\alpha \mid \alpha\in B)$ is the ramification type of $C$.
For convenience, we set $n_\alpha = 0$ for $\alpha\notin B$ and denote $\deg(f) = \sum_{\alpha\in B-\set{\infty}} n_\alpha$ by $N$.
Note that from congruence \eqref{eqn:typecong} it follows that $(\n, n_\infty) = (\n, N)$.

By examining the divisor of
${\prod_{\alpha\in k}(x-\alpha)^{j_\alpha} dx}/{y^i}$
for $0\leq i < \n$ and a choice of $j_\alpha\in \Z$ for $\alpha\in k$ with
only finitely many $j_\alpha \neq 0$ it can be seen that this differential form
lies in $\Diffs$ if and only if $j_\alpha$ satisfy
\begin{eqnarray}
\label{eqn:jmin}
(j_\alpha + 1) \n &>& i n_\alpha
\quad\textrm{for all $\alpha\in k$, and}\\
\label{eqn:jmax}
\left( \sum_{\alpha\in k} j_\alpha + 1 \right) \n &<& i N.
\end{eqnarray}
If we define
$$
j_{\min, \alpha, i} := \left\lfloor{i n_\alpha}/{\n} \right\rfloor,
$$
then condition \eqref{eqn:jmin} becomes $j \geq j_{\min, \alpha, i}$.
Note that $j_{\min, \alpha, i} \geq 0$, and that $j_{\min, \alpha, i}=0$ whenever
$n_\alpha = 0$ (i.e., $\alpha\notin B$).
Therefore it is possible to define polynomials
$$
h_{\min, i}(x) := \prod_{\alpha\in k} (x-\alpha)^{j_{\min,\alpha, i}}
$$
for each $0 \leq i < \n$. We have $\deg(h_{\min,i}) = \sum_{\alpha\in k} j_{\min, \alpha, i}$.
Condition \eqref{eqn:jmax}  can be rewritten as $\sum_{\alpha\in k} j_\alpha \leq j_{\max, i}$,
where
$$
j_{\max,i} :=
\left\lceil {i N}/{\n} \right\rceil - 2.
$$
We conclude that
${h(x) h_{\min, i}(x) dx}/{y^i} \in \Diffs$
for some $h\in k(x)$ if and only if $h\in k[x]$ and
$\deg(h) \leq j_{\max, i} - \deg(h_{\min,i}).$
For example, for $i=0$ this condition becomes $\deg(h) \leq -1$.
We set
$$
\omega_{j,1} :={h_{\min, i}(x) dx}/{y^i}.
$$
If we define subspaces $\cD_i \subset \Diffs$ by
$$
\cD_i := \left\{ h(x)\, \omega_{j,1}
\mid  h\in k[x],\ \deg(h) \leq j_{\max, i} - \deg(h_{\min,i}) \right\}
$$
then
\begin{eqnarray*}
\dim(\cD_i) &=&  j_{\max, i} - \deg(h_{\min,i}) + 1 \\
&=& \left\lceil{i N}/{\n} \right\rceil 
- \sum_{\alpha\in B-\set{\infty}} \left\lfloor {i n_\alpha}/{\n} \right\rfloor - 1\\
&=& \sum_{\alpha\in B} {i n_\alpha}/{\n}
- \sum_{\alpha\in B} \left\lfloor {i n_\alpha}/{\n} \right\rfloor - 1\\
&=&  \sum_{\alpha\in B} \left\langle {i n_\alpha}/{\n} \right\rangle - 1.
\end{eqnarray*}

By a counting argument, it can be shown that
$$
\Diffs = \bigoplus_{i=1}^{\n-1} \cD_i.
$$
and so $\cD_i$ are precisely the eigenspaces defined in Section \ref{sec:intro},
with $\cD_0 = \set{0}$.
Differential forms
\begin{equation}\label{eqn:basiselts}
\omega_{i, j} := x^{j-1} \omega_{j,1}
\end{equation}
with $1\leq j \leq d_i := \dim(\cD_i)$ form a basis of $\cD_i$.

%=========================================================================
\section{Cartier Operator}\label{sec:cartier}
%=========================================================================
In \cite{Cartier}, Cartier defines an operator
$\cC$ on the sheaf $\Om$ satisfying the following properties
\begin{enumerate}
\item $\cC(\omega_1 + \omega_2) = \cC(\omega_1) + \cC(\omega_2)$,
\item $\cC(h^p \omega) = h \cC(\omega)$,
\item $\cC(dz) = 0$,
\item $\cC(dz/z) = dz/z$,
\end{enumerate}
for all local sections $\omega$, $\omega_1$, $\omega_2$ of $\Om$
and $z$ of $\cO_C$.
The first two properties are collectively referred to as
$p^{-1}$-linearity. 
It follows that for a local section
$$
\omega = (h_0^p + h_1^p z + \ldots + h_{p-1}^p z^{p-1}) dz/z
$$
of $\Om$ with nonzero $z$, we have
\begin{equation*}%\label{eqn:five}
\cC(\omega) = h_0 dz/z.
\end{equation*}
This operator induces a $p^{-1}$-linear map
$$
\cC: \Diffs \to \Diffs.
$$
If one defines $\rank(\cC) := \dim_k \cC(\Diffs)$, then
$$
a_C = g_C - \rank(\cC).
$$
\smallskip

Define integer-valued functions $\e$ and $\s$ on $\set{1,\ldots,\n-1}$ as follows.
\begin{enumerate}
\item $1\leq \s(i) \leq \n-1$ is the unique integer satisfying $p\s(i) \equiv i \pmod \n.$
\item $0\leq \e(i) \leq p-1$ is  the unique integer satisfying $\n \e(i) \equiv -i \pmod p.$
\end{enumerate}
Then 
\begin{equation}\label{eqn:sigeps}
p\s(i) - \n\e(i) = i.
\end{equation}
Indeed, by the Chinese Remainder Theorem, $p\s(i) - \n\e(i) \equiv i \pmod{p\n}$,
and equality \eqref{eqn:sigeps} follows from the bounds on $\s(i)$ and $\e(i)$.

\begin{rmk}\label{rmk:infinity}
By composing $\pi: C\to \Proj$ with a suitable automorphism of $\Proj$
we can obtain a covering of $\Proj$ isomorphic to $C$ that is
not ramified over $\infty\in \Proj(k)$.
Since the $a$-number is an isomorphism invariant, we may assume, 
without loss of generality, that $\pi$ is not branched over $\infty$,
and therefore $n$ divides $N$.
Further, we make that assumption.
\end{rmk}

We will require the following result.
\begin{lemma}\label{lem:hunder}
For all $0<i<\n$, rational function 
$$
h_{\rnk, i} := f^{\e(i)} h_{\min,i}/ h_{\min, \s(i)}^p \in k(x)
$$
is a polynomial with
\begin{equation*}%\label{eqn:degree}
\deg(h_{\rnk, i}) = p d_{\s(i)} - d_i + p - 1
\end{equation*}
all of whose zeros have orders less than $p$.
\end{lemma}
\begin{proof}
Let $\alpha \in k$. Then
\begin{eqnarray*}
\ord_\alpha(h_{\rnk,i})
&=& 
\e(i) \ord_\alpha(f) + \ord_\alpha(h_{\min,i}) - p \ord_\alpha(h_{\min, \s(i)})\\
&=& \e(i) n_\alpha + j_{\min,\alpha,i} - p  j_{\min,\alpha,\s(i)}\\
&=& \e(i) n_\alpha + j_{\min,\alpha,p\s(i) - n\e(i)} - p  j_{\min,\alpha,\s(i)}\\
&=& \e(i) n_\alpha + \left\lfloor {(p\s(i) - n\e(i)) n_\alpha}/{\n} \right\rfloor
- p \left\lfloor {\s(i) n_\alpha}/{\n} \right\rfloor\\
&=& \left\lfloor {p\s(i) n_\alpha}/{\n} \right\rfloor
- p \left\lfloor {\s(i) n_\alpha}/{\n} \right\rfloor.
\end{eqnarray*}
Obviously, $0 \leq \ord_\alpha(h_{\rnk, i}) \leq p - 1,$ as claimed. 
Notice that if $n_\alpha = 0$, then
$\ord_\alpha(h_{\rnk, i})=0$. In other words,
polynomial $h_{\rnk, i}$ has a zero at $\alpha\in k$
only if $\alpha\in B-\set{\infty}$.

Finally, we know that $d_i = j_{\max, i} - \deg(h_{\min, i}) + 1,$ and therefore,
\begin{eqnarray*}
\deg(h_{\rnk, i})
&=& N\e(i) +  j_{\max, i} - d_i + 1 \\
&& - p (j_{\max, \s(i)} - d_{\s(i)} + 1)\\
&=& p d_{\s(i)} - d_{i}\\
&&  + N\e(i) + \left(\left\lceil {N (p\s(i)-n\e(i))}/{\n} \right\rceil - 1 \right)
 - p \left( \left\lceil {N \s(i)}/{\n} \right\rceil - 1\right)\\
&=& p d_{\s(i)} - d_i + p - 1.
\end{eqnarray*}
\end{proof}

\begin{cor}\label{cor:unitideal}
For $0<i<\n$, define polynomials $f_{i,0}, \ldots, f_{i, p-1} \in k[x]$ by decomposition
\begin{equation}\label{eqn:decomposition}
h_{\rnk, i}(x) = \sum_{t=0}^{p-1} f_{i, t}(x)^p x^t,
\end{equation}
where $h_{\rnk, i}\in k[x]$ is defined in Lemma \ref{lem:hunder}.
Then ideal $(f_{i,0}, \ldots, f_{i, p-1})$ is the unit ideal of $k[x]$.
\end{cor}
\begin{proof}
Indeed, a common root of $f_{i,0}, \ldots, f_{i, p-1}$ would be a root
of $h_{\rnk, i}(x)$ of order no less than $p$, contradicting Lemma \ref{lem:hunder}.
\end{proof}

\begin{thm}\label{thm:cartier}
Let $k$ be an algebraically closed field of characteristic $p>0$ and
$$
C: y^\n = f(x)
$$
be a Kummer cover of $\Proj$ defined $k$ 
with the roots of polynomial $f\in k[x]$ having order less than $\n$.
Let $\omega_{i,j}$ be given by \eqref{eqn:basiselts}, and polynomials 
$f_{i,0}, \ldots, f_{i, p-1}$ be defined by decomposition \eqref{eqn:decomposition}.
Then
$$
\cC\left(\omega_{i,j} \right) = x^{\lfloor j/p \rfloor} f_{i,  \langle -j/p \rangle p}(x)\, \omega_{\s(i), 1}.
$$
\end{thm}
\begin{proof}(after \cite{YuiJV})
By equality \eqref{eqn:sigeps}, $y^{p\s(i)-i} = y^{\n\e(i)} = f(x)^{\e(i)}$.
Therefore,
\begin{eqnarray*}
\cC\left(\omega_{i,j} \right)
&=&
\cC\left({x^{j-1} h_{\min, i}(x) dx}/{y^i} \right)\\
&=&
\cC\left( {x^{j-1} h_{\min, i}(x) f^{\e(i)}(x) dx}/{y^{p \s(i)}} \right)\\
&=&  
\frac{x^{\lfloor j/p \rfloor} h_{\min,\s(i)}(x)}{y^{\s(i)}}
\cC\Bigl(x^{\langle j/p \rangle p} h_{\rnk, i}(x)\, \frac{dx}{x}\Bigr)\\
&=& 
x^{\lfloor j/p \rfloor} f_{i,  \langle -j/p \rangle p}(x)\, \omega_{\s(i), 1}.
\end{eqnarray*}
\end{proof}

\begin{rmk}\label{rmk:choiceofmd}
Since $j\leq d_i$, we have $\lfloor j/p \rfloor \leq \lfloor d_i/p \rfloor$.
Also note that
\begin{eqnarray*}
\deg(f_{i,j}) &\leq &
\deg(f_{i, \langle (-d_i-1)/p\rangle p})\\
&=&
\lfloor \deg(h_{\rnk, i})/p\rfloor\\
 &=& d_{\s(i)} - \lfloor d_i / p\rfloor,
\end{eqnarray*}
for all $0 \leq j \leq p-1$,
since the leading term of $f_{i, \langle (-d_i-1)/p\rangle p}$
comes from the leading term of $h_{\rnk, i}$.
\end{rmk}

%=========================================================================
\section{Rank of the Cartier operator}\label{sec:rank}
%=========================================================================

In proving Theorem \ref{thm:main} we will require the following technical result.
\begin{lemma}\label{lem:rank}
Let $f_1, \ldots, f_r \in k[x]$ be polynomials such that $(f_1, \ldots, f_{r})$
is the unit ideal of $k[x]$, and let $m\geq 0$ be an integer. 
Put $d:=\max_{1\leq i \leq r}(\deg(f_i))$.
If we define a vector subspace $V$ of $k[x]$ by
$$
V := \set{f_1 g_1 + \ldots + f_{r} g_r \mid g_1, \ldots, g_r \in k[x],\ \deg(g_j) < m\text{ for all $1\leq j \leq r$}},
$$
then
$$
\dim(V) \geq \min(2m, m+d).
$$
\end{lemma}
\begin{proof}
Further, we assume, without loss of generality, that none of $f_i$ are zero,
and that $\deg(f_1) = d$.
For $1\leq i \leq r$, let $h_i\in k[x]$ be the greatest common divisor of
$f_1, \ldots, f_i$, i.e., the monic polynomial satisfying $(h_i)=(f_1, \ldots, f_i)$.
Also define vector subspaces
$$
V_i := \set{f_1 g_1 + \ldots + f_{i} g_i \mid g_1, \ldots, g_i \in k[x],\ \deg(g_j) < m\text{ for all $1\leq j \leq i$}}
$$
of $V$. Thus we obtain a flag $V_1 \subset \cdots \subset V_r=V$ of $V$.
Notice that statements of linear dependence in $V_i$ are equivalent to those of existence
of syzygies $f_1 g_1 + \ldots + f_{i} g_i = 0$ with $\deg(g_i) <m$ 
and not all of $g_i\in k[x]$ equal to zero.

We proceed by induction.
Obviously, $\dim(V_1) = m$. In addition, $h_1 = g_1 \in V_1$.
Suppose that $h_i \in V_i$. Then $V_i = \set{h_i g \mid g\in k[x],\ \deg(g)<\dim(V_i)}$.
There exist nonzero polynomials $u$ and $v$ of lowest degree such that $h_i u = f_{i+1} v$.
Since $(h_{i+1}) = (h_i, f_{i+1})$, we have
$(h_i/h_{i+1}) u = (f_{i+1}/h_{i+1}) v$ with $h_i/h_{i+1}$ and $f_{i+1}/h_{i+1}$
being relatively prime polynomials. 
By the minimality condition imposed on the degrees of $u$ and $v$, 
these polynomials are constant multiples of $f_{i+1}/h_{i+1}$ and $h_i/h_{i+1}$, 
respectively.
The space
$$
\set{f_{i+1} g \mid g\in k[x],\ \deg(g)< \deg(v) = \deg(h_i/h_{i+1})}
$$
intersects space $V_i$ trivially, since otherwise there would exist a
nonzero polynomial $g$ of degree less than that of $v$ such that
$f_{i+1} g \in V_i \subset (h_i)$, which would be a contradiction.
If $\deg(v) \geq m$, then subspaces
$$
W_{i+1} := \set{f_{i+1} g \mid g\in k[x],\ \deg(g)<m}
$$
and $V_i$ of $V$ intersect trivially.
Therefore, $V_{i+1} = V_i \oplus W_{i+1}$,
\begin{eqnarray*}
\dim(V) &\geq& \dim(V_{i+1}) \\
&=& \dim(V_i) + \dim(W_{i+1})\\
&\geq& \dim(V_1) + m\\
&=& 2m,
\end{eqnarray*}
and we are done.

On the other hand, if $\deg(v) < m$, then
$$
W_{i+1} := \set{f_{i+1} g \mid g\in k[x],\ \deg(g)<\deg(v)}
$$
intersects $V_i$ trivially, and $V_{i+1} = V_i \oplus W_{i+1}$
with 
\begin{eqnarray*}
\dim(V_{i+1}) &=& \dim(V_i) + \dim(W_{i+1})\\
&=&\dim(V_i) + \deg(v)\\
&=&\dim(V_i) + \deg(h_i) - \deg(h_{i+1})\\
&=&\dim(V_1) + \deg(h_1) - \deg(h_{i+1}),
\end{eqnarray*}
the last equality following by induction on $i$ and telescoping of the degrees.
In addition, $h_{i+1}$, which is a constant multiple of $v h_i$, lies in $V_{i+1}$.
If the latter outcome holds for every $1\leq i \leq r$, then
\begin{eqnarray*}
\dim(V) &=& \dim(V_r)\\
&=& \dim(V_1) + \deg(h_1) - \deg(h_r)\\
&=& m + d,
\end{eqnarray*}
since $h_1 = f_1$ and $h_r = 1$.
\end{proof}

\begin{rmk}\label{rmk:twopoly}
In Lemma \ref{lem:rank}, if $r=1$ or $2$, then
a strict equality $\dim(V) = \min(2m, m+d)$ holds.
Indeed, if $r=1$, then $f_1$ is a nonzero constant, $d := \deg(f_1) =0$,
and $\dim(V) = m = m+d$.

If $r=2$, let $u$ and $v$ be nonzero polynomials of smallest degree
such that $f_1 u = f_2 v$. 
Since $(f_1, f_2)=1$, $u$ and $v$ are constant multiples of $f_2$ and $f_1$,
respectively.
One has $V=V_2$, and by the inductive step of Lemma \ref{lem:rank}
either $m<d$ and $\dim(V_2)=2m$, or $d \leq m$ and
$\dim(V_2)= m + d$.
\end{rmk}

\begin{cor}\label{cor:individualrank}
With all terms as defined in the previous sections,
$$
\min\left(2 \lfloor{d_i}/{p} \rfloor, 
d_{\s(i)}
\right)
\leq \dim(\cC(\cD_i)) \leq 
\min(d_i, d_{\s(i)}).
$$
\end{cor}
\begin{proof}
Combine Lemma \ref{lem:hunder}, Corollary \ref{cor:unitideal}, Theorem
\ref{thm:cartier}, and Lemma \ref{lem:rank}.
Apply the latter to polynomials $f_{i,0}, \ldots, f_{i,p-1}$,
using $m = \lfloor d_i / p\rfloor$ and
$d = d_{\s(i)} - \lfloor d_i / p\rfloor,$
chosen so due to Remark \ref{rmk:choiceofmd}.

The second inequality follows from the fact that the dimension of the image of a
$p^{-1}$-linear map does not exceed the dimensions of the domain and codomain.
\end{proof}

\begin{proof}[Proof of Theorem \ref{thm:main}]
From direct sum decomposition \eqref{eqn:directsum} it follows that
\begin{equation*}
\rank(\cC) = \sum_{i=1}^{\n-1} \dim(\cC(\cD_i)).
\end{equation*}
Now apply Corollary \ref{cor:individualrank} and \cite[(5.2.8)]{LO}.
\end{proof}

\begin{proof}[Proof of Corollary \ref{cor:superspecial}]
One has
\begin{eqnarray*}
d_i
&=& \sum_{\alpha\in B} \langle {i n_\alpha}/\n \rangle - 1\\
&=& \sum_{\alpha\in B} \langle {(p\s(i) - \n\e(i)) n_\alpha}/{\n} \rangle - 1\\
&=& \sum_{\alpha\in B} \langle {p\s(i) n_\alpha}/{\n} \rangle - 1\\
&\leq& \sum_{\alpha\in B} p \langle {\s(i) n_\alpha}/{\n} \rangle - 1\\
&=& p d_{\s(i)} + p - 1,
\end{eqnarray*}
and therefore $\lfloor d_i/p \rfloor \leq d_{\s(i)}.$
Thus, if $d_i \geq p$, then $\min(2 \lfloor d_i/p \rfloor, d_{\s(i)}) \geq 1$
and $g_C - a_C \geq 1$.
\end{proof}

\begin{proof}[Proof of Corollary \ref{cor:chartwo}]
Recall that $p=2$.
If $d_i$ is even, then $2\lfloor d_i/p \rfloor = d_i$ and, by
Corollary \ref{cor:individualrank}, $\dim(\cC(\cD_i)) = \min(d_i, d_{\s(i)})$.

Suppose that $d_i$ is odd. Then $\langle (-d_i-1)/p\rangle p = 0$,
and therefore $\deg(f_{i,0}) \geq \deg(f_{i,1})$.
Apply Remark \ref{rmk:twopoly} to 
$f_{i,0}$ and $f_{i, 1}$ with $d = \deg(f_{i,0})$ and $m = (d_i+1)/2$.
The dimension of the vector space
$$
V = \langle f_{i,0} g_0 + f_{i,1} g_1 \mid g_i \in k[x], \deg(g_i) < m
\textrm{ for $i=0,1$} \rangle
$$
equals $\min(2m, m+d)$.
But $2m = d_i + 1$ and $m + d = (\lfloor d_i/p \rfloor+1) + (d_{\s(i)} - \lfloor d_i/p \rfloor) = d_{\s(i)} + 1$. Therefore,
\begin{eqnarray*}
\dim(\cC(\cD_i)) & \geq & \dim(V) - 1 \\
&=& \min(d_i, d_{\s(i)}).
\end{eqnarray*}
But by Corollary \ref{cor:individualrank}, $\dim(\cC(\cD_i)) \leq \min(d_i, d_{\s(i)})$.
We again obtain equality $\dim(\cC(\cD_i)) = \min(d_i, d_{\s(i)})$,
and the desired formula follows by direct sum decomposition \eqref{eqn:directsum}.
\end{proof}

%=========================================================================
% Bibliography
%=========================================================================

%=========================================================================
\end{document}